\documentclass[11pt, letterpaper
]{amsart}

\usepackage{amsfonts,graphics, amsthm, amsmath, amscd, amssymb}
\usepackage[draft=false]{hyperref}
\usepackage[T1]{fontenc}
\usepackage{ae}
\usepackage[all]{xy}
\usepackage{xcolor}
\usepackage{tikz-cd}
\usepackage{tikz}
\usetikzlibrary{patterns}
\usepackage{float}

\numberwithin{equation}{section} % Number equations by section
\newtheorem{theorem}{Theorem}[section] % Theorems numbered within sections
\newtheorem{lemma}{Lemma}[section] % Lemmas numbered within sections
\newtheorem{definition}{Definition}[section] % Definitions numbered within sections
\newtheorem{prop}{Proposition}[section]
\newtheorem{remark}[theorem]{Remark}

\numberwithin{equation}{section}
\theoremstyle{definition}

\newcommand\Z{\mathbb{Z}}
\newcommand\Q{\mathbb{Q}}

\renewcommand\rho{\varrho}

\title{Prime scattering geodesic theorem}

\author{Sudhir Pujahari}
\address{School of Mathematical Sciences, National Institute of Science Education and Research, Bhubaneswar, An OCC of Homi Bhabha National Institute,  P. O. Jatni,  Khurda 752050, Odisha, India.}
\email{spujahari@niser.ac.in/spujahari@gmail.com}
\author{Punya Plaban Satpathy}
\address{Visitor, School of Mathematical Sciences, National Institute of Science Education and Research, Bhubaneswar, An OCC of Homi Bhabha National Institute,  P. O. Jatni,  Khurda 752050, Odisha, India.}
\email{psatpathy81@gmail.com}

\newcommand{\Ha}{\mathbb{H}}
\newcommand{\M}{\mathcal{M}}
\newcommand{\Mm}{\mathcal{M}_c}
\newcommand{\Mi}{\mathcal{M}_{\infty}}
\newcommand{\R}{\mathbb{R}}
\begin{document}
\keywords{Scattering geodesic; Modular surface; Dirichlet $L$ function; Euler's totient function; Hyperbolic space.}
\subjclass[2000]{11N45, 11N13, 11F72}

\begin{abstract}
The modular surface, given by the quotient $\mathcal{M} = \Ha/\text{PSL}(2,\Z)$, can be partitioned into a compact subset $\Mm$ and an open neighborhood of the unique cusp in $\mathcal{M}$. We consider scattering geodesics in $\mathcal{M}$, first introduced by Victor Guillemin in \cite{Guillemin1976-xr} for hyperbolic surfaces with cusps. These are geodesics in $\mathcal{M}$ that lie in $\mathcal{M} \setminus \Mm$ for both large positive and negative times. Associated with such a scattering geodesic in $\mathcal{M}$, a finite \textit{sojourn time} is defined in \cite{Guillemin1976-xr}. In this article, we study the distribution of these scattering geodesics in $\mathcal{M}$ and their associated \textit{sojourn times}.  In this process, we establish a connection between the counting of scattering geodesics on the modular surface and the study of positive integers whose prime divisors lie in arithmetic progression. This article is the first such result for scattering geodesics.
\end{abstract}

\maketitle

%\bigskip

%\tableofcontents

\section{Introduction and main results} \label{results}

The study of geodesics has been a central theme in mathematics due to its connection with different branches of Mathematics and Physics.  

Inspired by Riemann's explicit formula, Selberg \cite{Selberg1956} introduced a zeta function and established his trace formula. Further inspired by the connection of the zeros of the zeta function and prime numbers, Selberg \cite{Selberg1956} used his mighty trace formula to count the lengths of closed geodesics on a compact Riemann surface, say $M=\mathbb{H}/\Gamma$. The set of lengths of closed geodesics on M is directly connected with the zeros of the Selberg zeta function.  
Using his zeta function, Selberg \cite{Selberg1956} proved a prime number type theorem for the number of oriented closed geodesics of bounded length. In addition, we know that the closed geodesics in $M$ are in one-to-one correspondence with the hyperbolic conjugacy classes of $\Gamma$. 
After the seminal result of Selberg, finer studies have been done in \cite{Avdispahic2009-io}, \cite{Buser1992}, \cite{Chatzakos2024-wo}, \cite{Huber1961}, \cite{Iwaniec1984}, ,\cite{mirzakhani}, \cite{Randol1977-po}, \cite{Sarnakthesis}, \cite{Sarnak1982class}, \cite{soundyoung}.
There are several generalisations of the prime geodesic theorem to different settings.  For example, see \cite{Hejhal1973}, \cite{Hejhal1983}.  For a more detailed account, see \cite{Chatzakos2024-wo}.

One of the key characteristics of Selberg's trace formula is its unique insight into the relationship between the discrete spectrum of the hyperbolic Laplacian and the lengths of closed geodesics on a compact surface. Selberg later extended this formula to the non-compact case, incorporating both the discrete spectrum and the determinant of the scattering matrix, which encodes information about the continuous spectrum (cf. \cite[Theorem 10.2]{Iwaniec2002}). However, unlike discrete eigenvalues, which correspond to closed geodesics, the scattering matrix has no direct geometric counterpart in the Selberg trace formula.
In a series of papers written in the 1970s, Lax and Phillips applied their scattering theory framework to the spectral decomposition of the Laplace operator on finite-area hyperbolic surfaces with cusps. They extended classical scattering methods to study the automorphic wave equation, providing a new proof for the meromorphic continuation of Eisenstein series and offered new insights into the Selberg trace formula and the spectral properties of the hyperbolic Laplacian(\cite{Lax1976},\cite{LaxPhillips1980}).

In 1977, Victor Guillemin wrote a paper (\cite{Guillemin1976-xr}) in which he discussed the asymptotic behavior of the scattering matrix arising from the automorphic wave equation, following the work of Lax and Phillips (see \cite{Lax1976}, \cite{LaxPhillips1980}). He considered a finite-area hyperbolic surface with cusps given by $X = \mathbb{H}/\Gamma$, where $\mathbb{H}$ is the upper half-plane, and $\Gamma$ is a discrete subgroup of $\text{PSL}(2,\mathbb{R})$.  

As part of his framework, Guillemin introduced the notion of a \textit{scattering geodesic} on such a hyperbolic surface and showed that these geodesics spend only a finite amount of time (called the \textit{sojourn time}) inside the compact core of the surface. One of the main results in \cite{Guillemin1976-xr} is a trace formula type theorem that relates these sojourn times to the entries of the scattering matrix (cf. \cite[Theorem 3]{Guillemin1976-xr}).
Guillemin's result on scattering geodesics was later generalized by Lizhen Ji and Maciej Zworski (cf. \cite{Ji2001-ii}) extended the notion of scattering geodesics to a $\Q$-rank one locally symmetric space $\Gamma \backslash X$. They proved a result (see \cite[Theorem 1,2]{Ji2001-ii}) that relates the \textit{sojourn times} of these scattering geodesics to the singularities of the Fourier transforms of the scattering matrices of $\Gamma \backslash X$.

We now review the fundamental aspects of scattering geodesics from \cite{Guillemin1976-xr}. Let us assume that $X$ has $n$ number of cusps labeled with $\kappa_1,\kappa_2,...,\kappa_n$ which are not equivalent. By a result of Siegel \cite[Ch. 1]{siegel1973topics} for any fixed real $a \gg 0$, $X$ can be partitioned into a compact subset $X_0:=X_0^a$ and a finite number of open sets $X_i:=X_i^a$, called \textit{Cusp neighborhoods} such that each $X_i$ is isometric to the following set in the upper half-plane.  \footnote{Note that the lines $\Re(z) = -1/2$ and $\Re(z) = 1/2$ are identified using the translation $z \mapsto z+1$, which we assume lies in $\Gamma$.}
\begin{equation*}
    \{z \in \Ha \mid -\frac{1}{2} \leq \Re(z) \leq \frac{1}{2}, \Im(z) > a \}.\footnote{For a complex number $z = x+iy$, $\Re(z):= x, \Im(z):= y$.}
\end{equation*}
A geodesic $\gamma := \gamma(t)$ in $X$ is called a \textit{scattering geodesic} if it remains in $X \setminus X_0$ for $|t| \gg 0$. If $\gamma(t)$ lies in $X_i$ for $t < t_1$ and in $X_j$ for $t > t_2$ with $t_1 < t_2 \in \R$, then we say that $\gamma$ is scattered from the $i$-th cusp neighbourhood to the $j$-th cusp neighbourhood. The \textit{sojourn time} associated with such a scattering geodesic $\gamma$ is the total time it spends in $X_0$, starting from the moment it first enters $X_0$ to the moment it leaves $X_0$ permanently.

The following observation has been made in [cf. Appendix B, Corollary 2 \cite{Guillemin1976-xr}].  \\
  {\it   A scattering geodesic has the property that for large
negative and positive times it corresponds to a vertical line in 
the figure below (i.e. after we have mapped the appropriate cusp neighborhoods onto
the standard cusp neighborhood exhibited in the figure below).}
\begin{figure}[h]
    \centering
    \begin{tikzpicture}
        % Draw the rectangle (wider version)
        \draw (-1,0) -- (-1,3);
        \draw (1,0) -- (1,3);
        \draw (-1,0) -- (1,0);

        % Add labels
        \node[left] at (-1,1.5) {\Large $x = -1/2$};
        \node[right] at (1,1.5) {\Large $x = 1/2$};
        \node[below right] at (1,0) {\Large $y = a$}; % Placed near the intersection of x=1 and y=a
    \end{tikzpicture}
   \caption{Standard cusp neighbourhood in the $\Ha$.\\[2pt]
    (In this figure, the lines $x = 1/2$ and $x = -1/2$ are identified.)}
\end{figure}
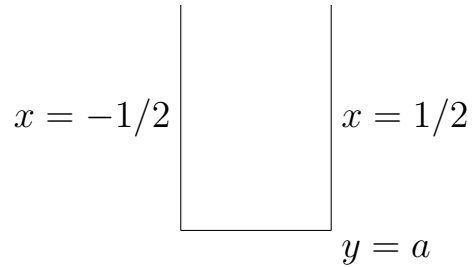
%\end{lemma}

In \cite[Theorem B1]{Guillemin1976-xr}, Guillemin proved the following result.

\begin{theorem}[Guillemin]
    There are countably many scattering geodesics that scatter between a given pair of cusps $\kappa_i$ and $\kappa_j$ in $X$.
\end{theorem}

In this paper, we prove a Prime number theorem kind result for Scattering geodesics.  This is the first such result in the theory of scattering geodesics.
More explicitly, we prove the following theorem.

\begin{theorem}\label{main}
    For $Y \gg 0$, let $\Pi(Y)$ be the number of distinct scattering geodesics in $\mathcal{M}:= \Ha/\text{PSL}(2,\Z)$ whose \textit{sojourn times} are bounded above by $\log Y$. Then, for sufficiently large $Y$,
    \begin{equation}
    \Pi(Y) = \frac{3Y}{2\pi^2T_0^2} + O_{T_0}(\sqrt{Y}(\log Y)^{2/3}(\log \log Y)^{1/3}), 
    \end{equation}
    where $T_0 \gg 0$ determines the size of the compact core of $\M$ (see \eqref{core}).
\end{theorem}
\begin{remark}
The problem of counting closed geodesics in compact manifolds bears a remarkable resemblance to counting prime numbers (see \cite{Chatzakos2024-wo} for more details). In our case, the counting of scattering geodesics on the modular surface is closely linked to the study of positive integers whose prime divisors lie in an arithmetic progression (see \S\ref{sss}).
\end{remark}

\subsection{The Modular surface $\M$}

Let $\Ha = \{z = x + iy \mid x, y \in \R, y > 0\}$ be the upper half-plane with the hyperbolic metric assigned,
$ds^2 = \frac{dx^2 + dy^2}{y^2}$. It is a well-known fact that any geodesic in $\Ha$ is either a half line of the form $\Re(z) = k$,
 or a half circle of the form $|z-p | = q$, with $p \in \R$, $q \in \R^{+}$ and $\Im(z) > 0$. A matrix \( A = \begin{pmatrix} a & b \\ c & d \end{pmatrix} \in \text{PSL(2,$\mathbb{Z}$)} \), $a,b,c,d \in \Z$ with $ad-bc=1$ acts on $\Ha$ via linear fractional transformations 
 \begin{equation*}
     z \mapsto \frac{az+b}{cz+d}.
 \end{equation*}
\newcommand{\f}{\mathcal{F}}
Note that all these transformations are isometries of $\Ha$ with respect to the hyperbolic metric \cite[Theorem 3.3.1]{Anderson2005-ox}. In this article, we work with the following fundamental domain $\f$ for the $\text{PSL}(2,\Z)$ action on $\Ha$.
\begin{equation}
    \f = \{z \in \Ha\mid 0 \leq \Re(z) \leq 1, |z| \geq 1, |z-1|\geq 1 \}.
\end{equation}

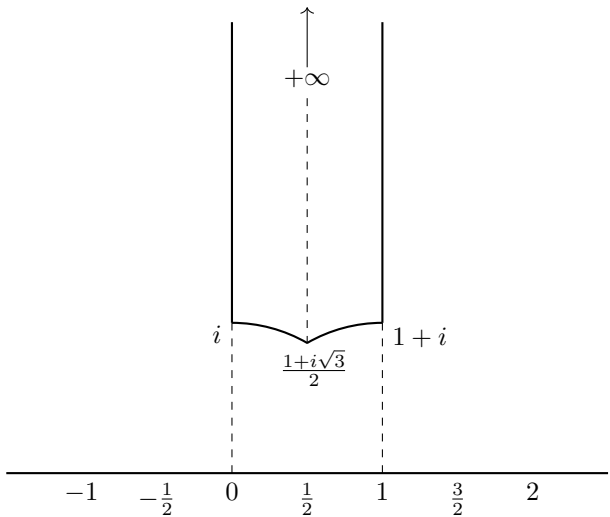
\begin{figure}[htbp]
\centering
\begin{tikzpicture}[scale=2.0, yscale=1.0]  
    \draw[black, thick] (0,1) arc[start angle=90, end angle=60, radius=1];
    
    \draw[black, thick, shift={(1,0)}] (-0.5,{sqrt(3)/2}) arc (120:90:1);
    
    \draw[dashed] (0,0) -- (0,1.0); 
    \draw[thick] (1.0,1.0) -- (1.0,3); 
    
    \draw[dashed] (1.0,0) -- (1.0,1.0); 
    \draw[thick] (0,1) -- (0,3); 
    
    \draw[thick] (-1.5,0) -- (2.5,0); 
    
    \node[below] at (-1,0) {\small $-1$};
    \node[below] at (-0.5,0) {\small $-\frac{1}{2}$};
    \node[below] at (0,0) {\small $0$};
    \node[below] at (0.5,0) {\small $\frac{1}{2}$};
    \node[below] at (1,0) {\small $1$};
    \node[below] at (1.5,0) {\small $\frac{3}{2}$};
    \node[below] at (2,0) {\small $2$};
    
    \node[above] at (-0.1,0.8) {\small $i$};
    \node[right] at (0.25,0.7) {\small $\frac{1 + i\sqrt{3}}{2}$};
    \node[right] at (1.0,0.9) {\small $1 + i$}; 
    
     \draw[dashed] (0.5,0.86) -- (0.5,2.5);
    \draw[->] (0.5,2.7) -- (0.5,3.1);
    \node[above] at (0.5,2.5) {\small $+\infty$};

\end{tikzpicture}
\caption{Fundamental domain for the PSL(2,$\mathbb{Z}$) action on $\mathbb{H}$.}
\end{figure}
The modular surface
is the quotient space $\mathcal{M} = \Ha/\text{PSL}(2,\Z)$, which is equipped with the metric induced from the upper half-plane. We know that the modular surface
has three singular points: a cusp and two ramification points of orders 2
and 3 at $i$ and $\frac{1+\sqrt{3}i}{2}$ respectively.
Note that even though $\mathcal{M}$ is not compact it has a finite area with respect to the measure induced by the hyperbolic metric on $\Ha$.
 Let $\pi : \Ha \longrightarrow \mathcal{M}$ be the
canonical projection, fixing a positive real $T_0 \gg 0$, we can partition $\mathcal{M}$ into two disjoint subsets $\mathcal{M}_c^{T_0}$ and 
$\mathcal{M}_{\infty}^{T_0}$ given by

\begin{equation} \label{core}
    \mathcal{M}_c^{T_0} = \pi({\mathcal{F}} \cap \{\Im(z) \leq T_0\}), \hspace{0.5cm} \mathcal{M}_{\infty}^{T_0} = \pi(\mathcal{F} \cap \{\Im(z) > T_0\}).
\end{equation}
Then, $\Mm^{T_0}$ is a compact subset of $\mathcal{M}$ which referred to as the \textit{compact core} for the rest of the article and $\Mi^{T_0}$ is an open neighbourhood of the single cusp in $\M$. It is a standard fact that geodesics in $\M$ are precisely the images of the geodesics in $\Ha$ under the natural projection map $\pi : \Ha \longrightarrow \mathcal{M}$.

We are interested in a class of scattering geodesic in $\M$ that runs between the cusp, which was first discussed by Victor Guillemin
in \cite{Guillemin1976-xr}.
\subsection{Scattering geodesics in $\M$}
We recall the following definition of scattering geodesics from \cite{Guillemin1976-xr}; 
\begin{definition}   A geodesic $\gamma(t)$ in $\M$ is called a scattering geodesic if it is contained in $\M \setminus \Mm^{T_0}$ for $|t| \gg 0$. The associated \textit{sojourn time} $l_{\gamma}$
is the total amount of time the geodesic spends in the compact core $\Mm^{T_0}$, starting from the time
it first entered $\Mm^{T_0}$ until the time it leaves $\Mm^{T_0}$ permanently.
\end{definition}
Next, we show that the scattering geodesics are in one-to-one correspondence with a subset of rationals in $[0,1)$.

\subsection{Classifying scattering geodesics in $\M$} We start with the following result;

\begin{theorem} \label{unqiuegeo}
For $w \in [0,1) \cap \Q$, let $\bar{\gamma}_w$ be the unique geodesic in $\Ha$ joining $w$ to $+\infty$. Given distinct $w_i = \frac{p_i}{q_i} \in  (0,1)\cap \Q$ with $\gcd(p_i,q_i) =1, q_i \geq 2$ for $i =1,2$, the geodesics $\bar{\gamma}_{w_1},\bar{\gamma}_{w_2}$ project onto the same geodesic in $\M$ if and only if the following conditions are satisfied
\begin{enumerate}
    \item $q_1 = q_2= q$.
    \item $q$ divides $p_1p_2+1$.
\end{enumerate}
\end{theorem}

\newcommand{\qq}{\mathcal{G}}
Fixing a $T_0$ as before, it follows from [Corollary 2,Theorem B1, \cite{Guillemin1976-xr}] that there are a countable number of non-trivial scattering geodesics in $\M$. Our first goal is to show that the set of scattering geodesics in $\M$ are in one-to-one correspondence with a certain subset of the interval $[0,1)$ (see \ref{qq}, Theorem \ref{scat}).

\begin{theorem} \label{scat}
 Let $\mathcal{S}$ be the set of scattering geodesics in $\M$, then there is a one-to-one correspondence between $\mathcal{S}$ and the set $\qq$ defined in \eqref{qq}, furthermore, for $w  = \frac{p}{q} \in \qq_q \subset \qq$ the corresponding unique scattering geodesic in $\M$ has an associated $\textit{sojourn time}$ equal to $2\log (qT_0)$. 
\end{theorem}

Next, we begin the construction of this set, which we denote by $\qq$.

\subsection{Construction of the set $\qq$} \label{constqq} For $q \geq 2$, consider the set \begin{equation*} 
    I_q := \{ p \in \mathbb{Z}^+ \mid 1 \leq p < q, \gcd(p, q) = 1 \}.
\end{equation*} The cardinality of \(I_q\) is precisely \(\phi(q)\).\footnote{For a natural number $n \in \mathbb{N}, \phi(n)$ is the Euler's totient function.} We also introduce a subset $S_q$ of $I_q$ defined as follows,

\begin{equation}\label{sq}
    S_q := \{ p\in \Z^+ \mid 1 \leq p < q, \, p^2 \equiv -1 \pmod{q} \}.
\end{equation}

For a given \(p \in I_q\), we now examine equation \(pp' \equiv -1 \pmod{q}\). Since \(\gcd(p, q) = 1\), there is a unique \(y_{p,q} \in I_q\) such that \(p y_{p,q} \equiv -1 \pmod{q}\).

    Now, there are two cases to consider:
    \begin{itemize}
        \item[(i)] \(p = y_{p,q}\),
        \item[(ii)] \(p \neq y_{p,q}\).
    \end{itemize}

    If \(p = y_{p,q}\), then \(p \in S_q\), where \(S_q\) is the set defined in \eqref{sq}. Let us now define the set
    \[
    C_q = \{ p \in I_q \mid p \neq y_{p,q} \}.
    \]
Observe that the elements of $C_q$ naturally form pairs, which we denote by  
$\boxed{p_1, p_2},$
where $p_1 \neq p_2$ and they satisfy the congruence relation  
$p_1 p_2 \equiv -1 \pmod{q}.$  
We now define a new subset of positive integers, denoted by $C^*_q$, which consists of the smaller integer from each such pair. Formally, we define  \begin{equation}
    C^*_q := \{ p \mid p = \min \{ p_1, p_2 \}, \boxed{p_1, p_2} \text{ is a pair in } C_q \}.
\end{equation}

Finally, we have the partition \(I_q = C_q \sqcup S_q\), where \(|C_q| = \phi(q) - s_q\) and \(|S_q| = s_q\). The cardinality of $C^*_q$ is precisely half of that of $C_q$, i.e. 
\begin{equation} \label{c*q}
    |C^*_q| = \frac{1}{2}(\phi(q)-s_q).
\end{equation}

\begin{definition} \label{qqq7}
  Let $\qq_{1} = \{0\}$ and for  \(q \geq 2, q \in \mathbb{Z}^+\), we define a subset \(\qq_q\) of the set of rational numbers as follows,
  \begin{equation}
      \qq_q := \bigg\{\frac{p}{q} \mid p \in C^*_q \sqcup S_q\bigg\}.
  \end{equation}
\end{definition}
\begin{remark}\label{orderqq}
    The set $\qq_q$ is ordered in an obvious way as a subset of rational numbers with standard ordering.
\end{remark}

\begin{definition} \label{defQ}
    Let $\qq \subset [0,1)$ be the following subset, 
    \begin{equation} \label{qq}
        \qq: = \displaystyle \bigsqcup_{q=1}^{\infty}\qq_q.
    \end{equation}
    \end{definition}
  \begin{remark} \label{ordQ}
      We order the elements of $\qq$ as follows, if $x \in \qq_{q_1}, y \in \qq_{q_2}$ with $q_1 < q_2$, then we declare $x < y$, where as if $x,y  \in \qq_q $ $x,y$ are ordered as in remark \ref{orderqq}.
  \end{remark}  
  \begin{figure}[h]
    \centering
    \includegraphics[width=0.8\textwidth]{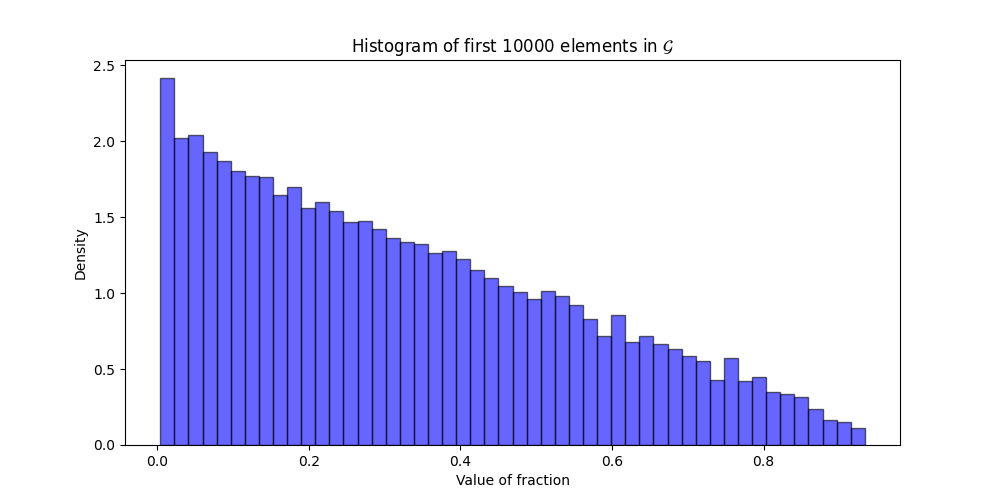}
    \caption{Density Histogram of first 10,000 elements in $\mathcal{G}$.}
    \label{fig:Q_plot}
\end{figure}

    It is clear from the definition that the cardinality of $\qq_q$ is the sum of the cardinalities of the sets $C^*_q$ and $S_q$.  Now using \eqref{c*q} and \eqref{sq}, we have the following result;

\begin{theorem} \label{cardqq}
    Let \(n_q\) denote the number of elements in \(\qq_q\). Then for $q \geq 2$, 
    \[
    n_q = \frac{1}{2}(\phi(q) + s_q),
    \]
    where \(\phi(q)\) is Euler's Totient function and \(s_q\) is the cardinality of the set $S_q$ defined in \eqref{sq}.
    
\end{theorem}

In the next section, we discuss preliminaries results.  In the final section, we prove the main results of the paper.

\section{Priliminaries}
In this section, we discuss the results required to prove the main theorems.

\subsection{Estimating $s_q$ and the sum $S(x) = \displaystyle\sum_{q \leq x} s_q$}

Recall that for $q \geq 2$, 

\begin{equation}
S_q = \{p \in \Z^+\mid  1\leq p <q \hspace{0.2cm} \text{ and } p^2 \equiv -1\pmod q\}
\end{equation}

and $s_q$ denotes the cardinality of $S_q$. We set $s_1 = 1$. 

\begin{lemma} \label{sq4k3}
If $q$ is a prime of the form $4k+3$, then $s_q = 0$.
\end{lemma}

\begin{proof}
This follows from the fact the value of the Legendre symbol $\left(\frac{-1}{p}\right)$ is $-1$ for a prime $p$ of the form $4k+3$.
\end{proof}

\begin{lemma}
If $q$ is a prime of the form $4k+3$, and $m = q^k$ with $k \geq 2$, then
$s_m = 0$.
\end{lemma}

\begin{proof}
Suppose we have a solution $y$ that satisfies $y^2 \equiv -1\pmod m$. This implies $y^2 \equiv -1\pmod {q^k}$, which further implies $y^2 \equiv -1\pmod q$. However, this is not possible by lemma \ref{sq4k3}.
\end{proof}
\begin{prop} \label{sq4k1}
Suppose $q$ is a prime of the form $4k+1$. Then for any $m \geq 1$, the equation $p^2 \equiv -1 \pmod {q^m}$ has exactly two distinct solutions modulo $q^m$, i.e., $s_{q^m} = 2$.
\end{prop}

\begin{proof}
Note that if $q^m$ divides $p^2+1$, then since $q$ is a prime, we must have $q$ dividing $p^2 +1$. So we start by investigating the base case $m=1$. For this case, the equation becomes $p^2 \equiv -1 \pmod q$. Since $-1$ is a quadratic residue modulo $q$, we know that this equation has exactly two solutions modulo $q$. We show that any such solution $x_0$ satisfying $x_0^2 \equiv -1 \pmod q$ can be lifted to a solution $x_1$ satisfying $x_1^2 \equiv -1 \pmod {q^m}$. 

 We perform this lifting inductively. Suppose we start with a solution $x_0$ modulo $q^j$ satisfying $x_0^2 \equiv -1 \pmod {q^j}$, and we construct a unique solution $x_1$ modulo $q^{j+1}$ such that $x_1^2 \equiv -1 \pmod {q^{j+1}}$ and $x_1 = x_0 + q^j y_0$.  

We begin by substituting the expression $x_1 = x_0 + q^j y_0$ into the equation $x_1^2 \equiv -1 \pmod {q^{j+1}}$ and get the following; 
\begin{align*}  
     2x_0 y_0 &\equiv -\frac{x_0^2+1}{q^j} \pmod {q}.
\end{align*}

Since $\gcd(2x_0, q) = 1$, the above equation yields a unique solution $y_0$ modulo $q$ by multiplying both sides of the last step by $(2x_0)^{-1}$ modulo $q$.  
\end{proof}
\begin{prop}
Let $q \in \Z^+$ be divisible by $n$, where $n$ is either 4 or a prime of the form $4k+3$. Then $s_q = 0$.
\end{prop}

\begin{proof}
Suppose $p$ satisfies $p^2 \equiv -1 \pmod q$. This implies $p^2 \equiv -1 \pmod{n}$. If $n = 4$, one can check by direct computation that $p^2 \equiv -1 \pmod{n}$ has no solution. The case where $n$ is a prime of the form $4k+3$ follows from Lemma \ref{sq4k3}.
\end{proof}

\begin{theorem}
Choose a $q \in \Z^+$ with one of the following prime factorizations,
\begin{equation*}
q = \prod\limits_{i=1}^{m} q_i^{n_i} \hspace{0.2cm} \text{ or } q = 2\prod\limits_{i=1}^{m} q_i^{n_i}.
\end{equation*}
Further, assume that all the primes $q_i$ are of the form $4k+1$. Then $s_q = 2^m$.
\end{theorem}

\begin{proof}
We prove this for the case where $q$ has a prime factorization of the first form above. We use the Chinese Remainder Theorem to explicitly construct a set of solutions to the congruence equation
\begin{equation*}
x^2 \equiv -1 \pmod q,
\end{equation*}
and also show that we have accounted for all possible solutions. For $i=1,2,...,m$, let $\{a_{i1},a_{i2}\}$ be the pair of distinct solutions modulo $q_i^{n_i}$ to the congruence equation  
\begin{equation*}  
    y^2 \equiv -1 \pmod {q_i^{n_i}}.
\end{equation*}  
Their existence follows from proposition \ref{sq4k1}. Now consider the equation  
\begin{equation*}  
    x^2 \equiv -1 \pmod q.
\end{equation*}  
This implies that $x$ must also satisfy the following congruences:  
\begin{equation*}  
    x^2 \equiv -1 \pmod {q_i^{n_i}}, \quad i=1,2,...,m .
\end{equation*}  
Thus, $x \equiv b_i \pmod{q_i^{n_i}}$, where $b_i$ is either $a_{i1}$ or $a_{i2}$.  

The system of equations  
\begin{equation}  
    x \equiv b_i \pmod{q_i^{n_i}}, \quad i=1,2,...,m  
\end{equation}  
has a unique solution $x$ modulo $q$, ensuring that $q$ divides $x^2 +1$. Since there are two distinct choices for each $b_i$, there are a total of $2^m$ distinct solutions to the equation $x^2 \equiv -1 \pmod{q}$. 

A similar argument applies when $q$ has the second form of factorization, with the only minor difference being that $x^2 \equiv -1 \pmod{2}$ has a unique solution.  
\end{proof}

\subsection{Estimating the sum $S(x)= \displaystyle \sum_{q \leq x} s_q$}\label{sss}
We start by summarizing the results we have proven so far regarding $s_q$. In what follows, $\omega(n)$ would denote the number of distinct prime factors of a natural number $n$ and $\mathcal{O} \subset \mathbb{N}$ is the subset of natural numbers such that $n \in \mathcal{O}$ if and only if any prime factor of $n$ is of the form $4k +1$.\footnote{We also include the number $1$ in $\mathcal{O}$}

\begin{theorem} \label{sqall}
    For $q \geq 2$, let $s_q$ be the cardinality of the following set
    \begin{equation*}
        \{p \mid  1 \leq p < q , p^2 \equiv -1 \pmod{q}\}.
    \end{equation*} Then we have,
    \begin{enumerate}
        \item $s_q = 0$ if $4|q$ or if $q$ is divisible by a prime of the form $4k+3$.
        \item $s_q = 2^{\omega(q)}$ if $q \in \mathcal{O}$.
        \item $s_q  = 2^{\omega(q/2)}$ if $q/2 \in \mathcal{O}$.
    \end{enumerate}  
    \end{theorem}
    \begin{lemma}
For $q \geq 1$, we have $s_q \leq \sqrt{q}$.
\end{lemma}

\begin{proof}
We can assume that $q \geq 3$, since for $q = 1,2$, the inequality holds by direct verification. Now, by theorem \ref{sqall}, $s_q = 0$ when $4|q$ or when $q$ has a prime factor of the form $4k+3$. In these cases, there is nothing to prove.

Thus, we need to consider the remaining two cases. We start by assuming that $q \in \mathcal{O}$ and has a prime factorization of the form
\begin{equation*}
q = \prod_{i=1}^m q_i^{n_i},
\end{equation*}
where each prime $q_i$ is of the form $4k+1$. In this case, we have
\begin{equation*}
s_q = 2^m < \prod_{i=1}^m \sqrt{q_i} \leq \prod_{i=1}^m \sqrt{q_i^{n_i}} = \sqrt{q}.
\end{equation*}

If instead $q/2 \in \mathcal{O}$, then we obtain $s_q = s_{q/2} < \sqrt{q/2} < \sqrt{q}$.
\end{proof}

\begin{theorem}
For $x \geq 1$, define the sum $S(x)$ as follows,
\begin{equation}
    S(x)  = \sum_{q\leq x} s_q.
\end{equation}
Then for $x\geq 5$ we have the following bounds for $S(x)$,
\begin{equation}
    2\lfloor \sqrt{x-1} \rfloor -1 \leq S(x)\leq \frac{2}{3}(x+1)^{3/2}.
\end{equation}
\end{theorem}
\begin{proof}
    We establish the lower bound first, observe that if $q  = k^2+1$ with $k \geq 2$, then $s_q \geq 2$ as there are atleast two distinct solutions to the equation
    $p^2 \equiv -1 \pmod{q}$, namely $p = \pm k$. The lower bound follows from counting the number of positive integers below $x$ which are of the form $k^2+1$.
    To get the upper bound we use the fact that $u \mapsto \sqrt{u}$ is an increasing function on $[0,\infty)$ from which we have,
    \begin{equation*}
        S(x)  = \sum_{q \leq x} s_q < \int_{0}^{x+1} \sqrt{u}du = \frac{2}{3}(x+1)^{3/2}.
    \end{equation*}
\end{proof}
We need an asymptotic estimate for $S(x)$.  To achieve that let's define the function $\tau(x)$ as follows,
    \begin{equation*}
        \tau(x):= \sum_{\substack{n \leq x \\ n \in \mathcal{O}}} 2^{\omega(n)} .
    \end{equation*}
    Then it is easy to verify that
  \begin{equation} \label{sntaurel}
      S(x) =\tau(x) +\tau\bigg(\frac{x}{2}\bigg).
  \end{equation}  
  The next result discusses the asymptotics of $\tau(x)$ for large $x$.
  \newcommand{\kk}{\mathcal{O}}
  \newcommand{\st}{\sigma_0}
\begin{theorem} \label{taues}
    For sufficiently large $x$, we have
    \begin{equation}
        \tau(x) = \sum_{\substack{n \leq x \\ n \in \kk}} 2^{\omega(n)}  = \frac{x}{\pi} +O \bigg ( \frac{x}{\log x}\bigg).
    \end{equation}
\end{theorem}
\begin{proof}
    We start by looking at the $L$-function,
    \begin{equation}
        F(s):= \sum_{n \in \kk} \frac{2^{\omega(n)}}{n^s}.
    \end{equation}
    As $2^{\omega(n)} < d(n)$\footnote{This can be easily verified by writing the prime factorization of $n$.}, where $d(n)$ is the divisor function and the series $\displaystyle \sum_{n=1}^{\infty} \frac{d(n)}{n^s}$ converges absolutely for $\Re(s) > 1$, we have that the series defining $F(s)$ converges absolutely for $\Re(s) > 1$. Hence $F(s)$ is holomorphic in this region (cf. \cite[Theorem 11.12]{Apostol2010-pc}) and we have the Euler product expansion (cf. \cite[Theorem 11.7]{Apostol2010-pc}), where the product is over all the primes of the form $4k+1$.
    \begin{align*}
        F(s) & = \prod_{p \equiv 1 \bmod{4}}\bigg( \sum_{k=0}^{\infty} \frac{2^{\omega(p^k)}}{p^{ks}}  \bigg) \\
        & = \prod_{p \equiv 1 \bmod{4}} \bigg( 1+\sum_{k=1}^{\infty} \frac{2}{p^{ks}}  \bigg) \\
        & = \prod_{p \equiv 1 \bmod{4}} \bigg( 1+ \frac{2p^{-s}}{1-p^{-s}}  \bigg) \\
        & = \prod_{p \equiv 1 \bmod{4}} \frac{1+p^{-s}}{1-p^{-s}}  \\
        & = (1+2^{-s})^{-1}\prod_{p} (1+p^{-s}) \bigg( \prod_{p \equiv {1}\bmod{4} } (1-p^{-s})^{-1} \prod_{p \equiv 3 \bmod{4}} (1+p^{-s})^{-1} \bigg )\\
        & = (1+2^{-s})^{-1}\prod_p\frac{1-p^{-2s}}{1-p^{-s}}L(s,\chi) = \frac{\zeta(s) L(s,\chi)}{(1+2^{-s})\zeta(2s)},
    \end{align*}
where $\zeta(s)$ is the Riemann zeta function and  $L(s,\chi)$ is the Dirichlet's L-function associated to the non-principal character $\chi$ modulo 4 defined by 
 \[
\chi(n) =
\begin{cases}
  0, & \text{if } n \text{ is even}, \\
  1, & \text{if } n \equiv 1 \bmod{4}, \\
  -1, & \text{if } n \equiv 3 \bmod{4}.
\end{cases}
\]
 As $L(s,\chi)$ is an entire function (cf. \cite[Theorem 12.5]{Apostol2010-pc}) using the properties of $\zeta(s)$ we know that the function $F(s)$ is meromorphic in the region $\Re(s) \geq 1$ having no poles except for one simple pole at $s = 1$ with residue $C$, where\footnote{Note that $L(1,\chi) = 1-\frac{1}{3}+\frac{1}{5}-\frac{1}{7}+... = \arctan 1 = \frac{\pi}{4}$. }
\begin{equation} \label{conc}
    C =\frac{L(1,\chi)}{(3/2)(\pi^2/6)} = \frac{4L(1,\chi)}{\pi^2} = \frac{1}{\pi}.
\end{equation}
Then applying the Wiener–Ikehara Tauberian Theorem as in \cite[Theorem 1.2]{Murty2024-zp} we conclude that for sufficiently large $x$,
\begin{equation}
     \tau(x) = \sum_{\substack{n \leq x \\ n \in \kk}} 2^{\omega(n)} = \frac{x}{\pi} +O\bigg(\frac{x}{\log x}\bigg).
\end{equation}

\end{proof}
Combining the results of Theorem \ref{taues} and\eqref{sntaurel}, we have the following result.
\begin{theorem} \label{Snfinal}
    For sufficiently large $x$, we have
  \begin{equation*}\label{E2}
      S(x) = \sum_{q \leq x}s_q  = \frac{3x}{2\pi} +O \bigg (\frac{x}{\log x} \bigg).
  \end{equation*}  
\end{theorem}

\subsection{A counting function for the sojourn time of scattering geodesics} \label{count}

Here we estimate the sum $\Psi(x):= \displaystyle \sum_{q \leq x} n_q$, where $n_q = |\qq_q|$ with the set $\qq_q$ as defined in \eqref{qqq7}. Recall that we have the equality;

\begin{equation}
    n_q = \frac{1}{2}(\phi(q)+s_q) , s_q = |\{p \mid 1 \leq p < q, p^2 \equiv {-1} \bmod {q}\}|.
\end{equation}
From \cite[Theorem 1.1]{Liu2016-bu} we have 
\begin{equation}\label{E1}
    \sum_{q\leq x} \phi(q) = \frac{3}{\pi ^2}x^2 +O(x(\log x)^{2/3}(\log \log x)^{1/3}).
\end{equation}
for sufficiently large $x$. 
Combining \eqref{E1} along with Theorem \ref{Snfinal} we get the following result;
\begin{theorem}\label{psi}
    For sufficiently large $x$, we have,
    \begin{equation}
        \Psi(x) = \frac{3x^2}{2\pi^2} +O(x(\log x)^{2/3}(\log \log x)^{1/3}).
    \end{equation}
  
\end{theorem}

\section{Proofs}

\subsection{Proof of Theorem \ref{unqiuegeo}:} 
    The geodesics $\bar{\gamma}_{w_1},\bar{\gamma}_{w_2}$ project onto the same geodesic in $\M$ if and only if $\exists \ \sigma \in \text{PSL}(2,\Z)$ such that
    $\sigma \bar{\gamma}_{w_1} = \bar{\gamma}_{w_2}$. In case $\sigma$ maps $\infty$ to $\infty$ and $w_1$ to $w_2$, this forces $\sigma$ to be a translation i.e. $\sigma(z) = z+n$ for some $n  \in \Z$ and then the constraint  $w_1,w_2 \in (0,1)$ forces $n$ to be 0 which in turn implies $w_1 = w_2$. 

    So for some \( \sigma = \begin{pmatrix} a & b \\ c & d \end{pmatrix} \in \text{PSL(2,$\mathbb{Z}$)} \), $a,b,c,d \in \Z$ with $ad-bc=1$ such that
    $\sigma(\infty) = w_2$ and $\sigma(w_1) = \infty$. Note that we must necessarily have $c\neq 0$. From the two conditions we get
    \begin{equation}
        \frac{a}{c} = \frac{p_2}{q_2} , \frac{-d}{c} = \frac{p_1}{q_1}.
    \end{equation}
    This gives us $a = kp_2, d= -mp_1$ and $c = kq_2 = mq_1$ for some $k,m \in \Z \setminus \{0\}$. Since we have $ad-bc =1 \implies -kmp_1p_2 -bc =1$. Since both $k,m$ divide $c$ this forces $|k| = |m|=1$, as both $q_1,q_2$ are positive we must have $k = m= \pm 1$, without loss of generality we can assume $k = m=1$.
    From this we get $a = p_2, d = -p_1$ and $c = q_1=q_2 = q$. As $ad-bc =1,$ we have $-p_1p_2 = 1+bq.$
    As a result we conclude that $q$ divides $(1+p_1p_2)$. 

    Conversely, assume that $q_1 = q_2 = q$ and that $q$ divides $p_1p_2+1$. Consider the two by two matrix \( \sigma = \begin{pmatrix} p_2 & b \\ q & -p_1 \end{pmatrix} \in \text{PSL(2,$\mathbb{Z}$)} \) where $b = -\frac{1+p_1p_2}{q}$, then it can be verified that $\sigma(\infty) = \frac{p_2}{q} =w_2$ and
    $\sigma(w_1) = \sigma(\frac{p_1}{q}) = \infty$. 

\begin{remark} \label{sc012}
Observe that for $ w \in (0,1) \cap \Q $ the two geodesics $\bar{\gamma}_w,\bar{\gamma}_0$ in $\Ha$ project onto two distinct geodesics in $\M$.
\end{remark}

\subsection{Proof of Theorem \ref{scat}}
    Let $\gamma(t)$ be a scattering geodesic in $\M$, arguing as in \cite{Guillemin1976-xr} we can pick a lift $\bar{\gamma}(t)$ in $\Ha$ of the geodesic $\gamma(t)$ such that $\bar{\gamma}(-\infty) = +\infty$ and $\bar{\gamma}(+\infty) = w = \frac{p}{q}\in \Q$. Furthermore, applying a $B \in \text{SL}(2,\Z)$ if necessary we can move $\bar{\gamma}$ into the strip \begin{equation*}
        \{z \in \Ha \mid 0 \leq \Re(z) < 1\}.
    \end{equation*} 
    That is we can choose $w$ to lie in the set $[0,1) \cap \Q$ so that $\bar{\gamma} = \bar{\gamma}_w$.
    If $w = 0,$ then $w \in \qq_1$, and in this case $\gamma$ corresponds to the unique element $0 \in \qq_1 \subset \qq$ as follows from the remark \ref{sc012}. So we can assume that $w = \frac{p}{q}$ with $q \geq 2$, $1 \leq p < q$ and $\gcd(p,q) =1$. Then from Theorem \ref{unqiuegeo} and the discussion in the beginning of \S \ref{constqq} we see that exactly one of the following two cases can occur;

        { Case I:} $p^2 \equiv -1 \pmod{q}$.  So we have $p \in S_q$ and consequently $w = \frac{p}{q} \in \qq_q \subset \qq$ is the unique element of $\qq$ associated to the scattering geodesic $\gamma$ in $\M$.  
        
        {Case II:} $\exists p' \neq p$ such that $1 \leq p' < q$ and $pp' \equiv -1 \pmod{q}$. It follows from Theorem \ref{unqiuegeo} that the geodesics $\bar{\gamma}_w$ and $\bar{\gamma}_{w'}$ project onto the same geodesic $\gamma \in \M$, where $w' = \frac{p'}{q}$. If we let $p^* = \text{min}\{p_1,p_2\}$ and
        set $w^* = \frac{p^*}{q}$, then $w^* \in \qq_q$ and $\bar{\gamma}_{w^*}$ maps down to $\gamma$ in $\M$. In this case $w^* \in \qq_q \subset \qq$ is the unique element of $\qq$ associated to the scattering geodesic $\gamma$ in $\M$.

    Conversely, for $w \in \qq \subset [0,1)$ let $\bar{\gamma}_w$ be the unique geodesic $\Ha$ with end points at $+\infty$ and $w$, then it is clear from Theorem \ref{unqiuegeo} and the discussion in \S \ref{constqq} that for different values of $w$ chosen in $\qq$, the images of $\bar{\gamma}_w$ correspond to distinct scattering geodesics in $\M$.

\begin{figure}[htbp]
\centering
\begin{tikzpicture}[scale=2.0, yscale=1.0] 
    \draw[black, thick] (0,1) arc[start angle=90, end angle=60, radius=1];
    
    \draw[black, thick, shift={(1,0)}] (-0.5,{sqrt(3)/2}) arc (120:90:1);

    \draw[dashed] (0,0) -- (0,1.0); 
    \draw[thick] (1.0,1.0) -- (1.0,3); 

     \draw[thick] (0,2) -- (1.0,2); 
    
    \draw[dashed] (1.0,0) -- (1.0,1.0); 
    \draw[thick] (0,1) -- (0,3); 
    
    \draw[thick] (-1.5,0) -- (2.5,0); 
    
    \node[below] at (-1,0) {\small $-1$};
    \node[below] at (-0.5,0) {\small $-\frac{1}{2}$};
    \node[below] at (0,0) {\small $0$};
    \node[below] at (0.5,0) {\small $\frac{1}{2}$};
    \node[below] at (1,0) {\small $1$};
    \node[below] at (1.5,0) {\small $\frac{3}{2}$};
    \node[below] at (2,0) {\small $2$};
     \node[right] at (1.0,2) {\small $\Im(z) = T_0$};
    \node[above] at (-0.1,0.8) {\small $i$};
    \node[right] at (0.25,0.7) {\small $\frac{1 + i\sqrt{3}}{2}$};
    \node[right] at (1.0,0.9) {\small $1 + i$}; 
    
     \draw[dashed] (0.5,0.86) -- (0.5,2.5);
    \draw[->] (0.5,2.7) -- (0.5,3.1);
    \node[above] at (0.5,2.5) {\small $+\infty$};
    
    \draw[thick, red] (0.3,1) -- (0.3,3); 
    \node[above] at (0.3,3) {$\bar{\gamma}_w$};
    \draw[thick, red, dotted] (0.3,0) -- (0.3,1); 
    \node[below] at (0.3,0) {\small $w$};
\end{tikzpicture}
\caption{$\bar{\gamma}_w$:= A lift in $\mathbb{H}$ of a scattering geodesic $\gamma$ in $\M$.}
\end{figure}
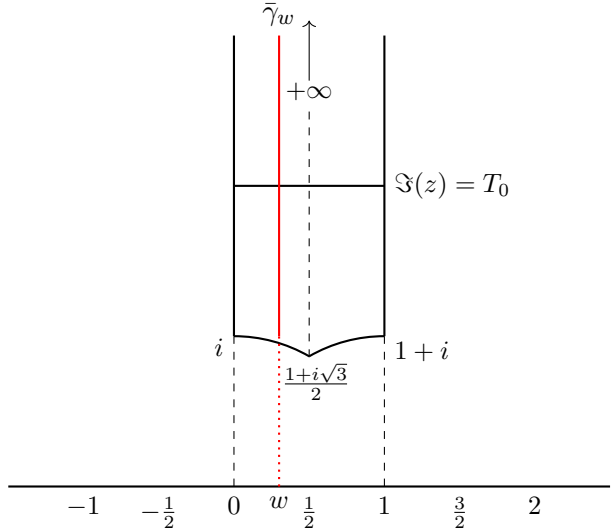

We now pick a scattering geodesic $\gamma(t)$ in $\M$ and calculate the associated sojourn time $l_{\gamma}$. As before, let $\bar{\gamma}_w(t)$ be the unique lift in $\Ha$ of the geodesic $\gamma$ such that $\bar{\gamma}_w(-\infty) = +\infty$ and $\bar{\gamma}_w(+\infty) = w = \frac{p}{q} \in  \qq_q \subset \qq$. Note that $\gamma$ enters the compact core $\Mm^{T_0}$ when $\bar{\gamma}_w$ intersects the line $\Im(z) = T_0$, that is, it enters at the point $z_0 = \frac{p}{q}+iT_0$. 

\noindent We assume $w \neq 0$ and pick a matrix \( B = \begin{pmatrix} a & b \\ q & -p \end{pmatrix} \in \text{SL(2,$\mathbb{Z}$)} \) where $a, b \in \Z$ with $1 \leq a < q$ satisfying $ap \equiv -1 \pmod{q}$ and $b = -(1+ap)q^{-1}$.\footnote{If $w = 0 \implies q=1,p=0$. In this case we choose $B$ as the map $z \mapsto -\frac{1}{z}$ and the rest of the argument follows.} Then we have $B(w) = \infty$ and arguing as in \cite{Guillemin1976-xr} we see that $\gamma$ leaves the core $\Mm^{T_0}$ when the geodesic $B\bar{\gamma}_w$ meets the line $\Im(z) = T_0$, that is, at the point $z_1 = \frac{p}{q}+i\frac{1}{T_0q^2}$. 

\noindent We identify $\M$ with the fundamental domain 
$ \mathcal{F}$ (where of course sides of $\mathcal{F}$ are identified by appropriate elements in $\text{PSL}(2,\Z)$) and consider the vertical line segment, 
\begin{equation*}
   \bigg \{\frac{p}{q}+is, \frac{1}{T_0q^2} < s < T_0\bigg\}.
\end{equation*}
Then there exists an $s_0 \in \bigg( \frac{1}{T_0q^2} , T_0 \bigg)$ such that the portion of the scattering geodesic $\gamma$ that lies inside the compact core 
$\Mm^{T_0}$ is 
\begin{equation*}
   \bigg \{\frac{p}{q}+is, s_0 \leq s < T_0\bigg\} \cup  B \bigg( \bigg \{\frac{p}{q}+is, \frac{1}{T_0q^2} < s < s_0\bigg\} \bigg).
\end{equation*}
Since $B$ acts by isometry on $\Ha$, the sojourn time $l_{\gamma}$ is just sum of the hyperbolic lengths of the two vertical line segments
\begin{equation*}
    \bigg \{\frac{p}{q}+is, \frac{1}{T_0q^2} < s < s_0\bigg\} \text{   and   }   \bigg \{\frac{p}{q}+is, s_0 \leq s < T_0\bigg\}.
\end{equation*}
As a result we get $l_{\gamma} = 2 \log (qT_0)$.\footnote{The argument involving the calculation of \textit{sojourn time} is adapted from \cite{Guillemin1976-xr}.}

\newcommand{\p}{\varphi}

\subsection{Proof of Theorem \ref{main}:}
Fix a $Y \gg 0$, and let $\Pi(Y)$ be the number of scattering geodesics in $\M$ whose sojourn time is bounded above by $\log Y$, i.e.
\begin{equation}
    \Pi(Y) = |\{\gamma \in \mathcal{S} \mid e^{l_{\gamma}} \leq Y\}|.
\end{equation}
Note that if $\gamma \in \mathcal{S}$, corresponds to a rational $\frac{p}{q} \in \qq_q \subset \qq$, then it's associated sojourn time $l_{\gamma} = 2 \log (qT_0)$, then $e^{l_{\gamma}}=(qT_0)^2 \leq Y.$
So we have

$ q \leq \frac{1}{T_0}\sqrt{Y}$. Hence,
\begin{equation}
    \Pi(Y) = \Psi \bigg(\frac{\sqrt{Y}}{T_0}\bigg).
\end{equation}
Now using Theorem \ref{psi}, we get the required result.

\textbf{Acknowledgement:}
This project was carried out entirely during the second author’s visit to NISER. The second author sincerely thanks the NISER for its hospitality and for providing a conducive research environment. Special thanks is also extended to Professor Brundaban Sahu for partly hosting the visit. The first author acknowledges support from the Science and Engineering Research Board [SRG/2023/000930].

 %\bibliographystyle{amsplain} 
%\bibliography{references}    
\providecommand{\bysame}{\leavevmode\hbox to3em{\hrulefill}\thinspace}
\providecommand{\MR}{\relax\ifhmode\unskip\space\fi MR }
\providecommand{\MRhref}[2]{%
  \href{http://www.ams.org/mathscinet-getitem?mr=#1}{#2}
}
\providecommand{\href}[2]{#2}

\end{document}